\documentclass[11pt]{amsart}
\usepackage[mathscr]{eucal}
\usepackage{amssymb, amsmath,array, amscd}
\usepackage{enumerate,verbatim}
\usepackage{graphicx}
\usepackage{url}
\usepackage[plainpages]{hyperref}

\input xy
\xyoption{all}

\newtheorem{theorem}{Theorem}[section]

\newtheorem*{claim}{Claim}
\newtheorem{corollary}[theorem]{Corollary}
\newtheorem{lemma}[theorem]{Lemma}
\newtheorem{proposition}[theorem]{Proposition}

\theoremstyle{definition}
\newtheorem{definition}[theorem]{Definition}
\newtheorem{example}[theorem]{Example}
\newtheorem{remark}[theorem]{Remark}

\numberwithin{equation}{section}

\newcommand{\set}[1]{\left\{#1\right\}}

\renewcommand{\b}[1]{\mathbf{#1}}   
\renewcommand{\k}{\Bbbk}

\newcommand{\A}{\mathcal{A}}

\newcommand{\cJ}{\mathcal{J}}

\newcommand{\C}{\mathbb{C}}
\newcommand{\Q}{\mathbb{Q}}
\newcommand{\Z}{\mathbb{Z}}

\newcommand{\ab}{{\mathfrak{a}}}
\newcommand{\hh}{{\mathfrak{h}}}
\newcommand{\sC}{{\sf{C}}}
\newcommand{\sK}{{\sf{K}}}
\newcommand{\sd}{{\sf{d}}}
\newcommand{\sE}{{\sf{E}}}
\newcommand{\sI}{{\sf{I}}}
\newcommand{\sJ}{{\sf{J}}}
\newcommand{\sA}{{\sf{A}}}
\newcommand{\bQ}{{\mathbf{Q}}}

\newcommand{\PS}{{P\!\varSigma}}

\DeclareMathOperator{\id}{id}
\DeclareMathOperator{\Aut}{Aut}

\DeclareMathOperator{\Tor}{Tor}

\DeclareMathOperator{\tc}{{\sf TC}}
\DeclareMathOperator{\cl}{\sf cl}
\DeclareMathOperator{\zcl}{\sf zcl}
\DeclareMathOperator{\secat}{secat}
\DeclareMathOperator{\rank}{rank}
\DeclareMathOperator{\inn}{in}
\DeclareMathOperator{\geomdim}{geom dim}
\DeclareMathOperator{\IA}{IA}
\DeclareMathOperator{\GL}{GL}

\begin{document}

\title[Cohomology of almost-direct products]
{Cohomology rings of almost-direct products of free groups}

\author[Daniel C. Cohen]{Daniel C. Cohen$^\dag$}
\address{Department of Mathematics, Louisiana State University,
Baton Rouge, LA 70803, USA}
\email{\href{mailto:cohen@math.lsu.edu}{cohen@math.lsu.edu}}
\urladdr{\href{http://www.math.lsu.edu/~cohen/}
{http://www.math.lsu.edu/\char'176cohen}}
\thanks{{$^\dag$}Partially supported 
by National Security Agency grant H98230-05-1-0055}

\subjclass[2000]{
20F28,
20F36,
20J06,
55M30
}

\keywords{almost-direct product, cohomology ring, topological complexity}

\begin{abstract}
An almost-direct product of free groups is an iterated semidirect product of finitely generated free groups in which the action of the constituent free groups on the homology of one another is trivial.  We determine the structure of the cohomology ring of such a group.  This is used to analyze the topological complexity of the associated Eilenberg-Mac\,Lane space.
\end{abstract}


\maketitle

\section{Almost direct products of free groups} \label{sec:intro}

If $G_1$ and $G_2$ are groups, and $\alpha\colon G_1\to \Aut(G_2)$ is a homomorphism from $G_1$ to the group of (right) automorphisms of $G_2$, the semidirect product $G=G_2 \rtimes_\alpha G_1$ is the set $G_2 \times G_1$ with group operation
$(g_2,g_1)\cdot (g_2',g_1')=(\alpha(g_1')(g_2) g_2',g_1g_1')$. There is a corresponding split, short exact sequence  
\[
\xymatrix{1 \ar@/{}/[r] & G_2 \ar@/{}/[r]_<<{\ \ \iota_2} & G \ar@/{}/[r]_<<{\ \ \ \pi} & G_1 \ar@/{}/[r] \ar@/{}_{.6pc}/[l]_{\iota_1}  & 1,}
\]
where $\iota_1(g_1)=(1,g_1)$, $\iota_2(g_2)=(g_2,1)$, and $\pi(g_2,g_1)=g_1$.  Identifying $G_1$ and $G_2$ with their images under $\iota_1$ and $\iota_2$, the group $G$ is generated by $G_1$ and $G_2$.  Furthermore, for $g_1\in G_1$ and $g_2\in G_2$, the relation $g_1^{-1} g_2^{} g_1^{}=\alpha(g_1^{})(g_2^{})$ holds in $G$. If $G_1$ and $G_2$ are free groups, these are the only relations in $G$.

An almost-direct product of free groups is an iterated semidirect product
\[
G=\rtimes_{i=1}^\ell F_{n_i}=F_{n_\ell} \rtimes_{\alpha_\ell}(F_{n_{\ell-1}}\rtimes_{\alpha_{\ell-1}}( \cdots 
\rtimes_{\alpha_3} (F_{n_2} \rtimes_{\alpha_2} F_{n_1})))
\]
of finitely generated free groups in which the action of the group $\rtimes_{i=1}^j F_{n_i}$ on $H_1(F_{n_k};\Z)$ is trivial for each $j$ and $k$, $1 \le j < k \le \ell$.  In other words, the automorphisms $\alpha_k\colon \rtimes_{i=1}^{k-1} F_{n_i} \to \Aut(F_{n_k})$ which determine the iterated semidirect product structure of $G$ are $\IA$-automorphisms, inducing the identity on the abelianization of $F_{n_k}$.  If $F_{n_i}$ is freely generated by $x_{i,p}$, $1\le p \le n_i$, the group $G$ is generated by these elements (for $1\le i \le \ell$), and has defining relations  
\begin{equation} \label{eq:rels}
x_{i,p}^{-1} x_{j,q}^{} x_{i,p}^{} = \alpha_j(x_{i,p}^{})(x_{j,q}^{}), \quad 1\le i<j\le \ell,\quad 1\le p \le n_i, \quad 1\le q \le n_j.
\end{equation}

\begin{example} \label{ex:braid1}
Perhaps the most famous example of an almost-direct product of free groups is the Artin pure braid group $P_\ell$, the fundamental group of the configuration space $F(\C,\ell)$ of $\ell$ ordered points in $\C$.  The almost-direct product structure of $P_\ell=F_{\ell-1} \rtimes_{\alpha_{\ell-1}} \rtimes \cdots \rtimes_{\alpha_2} \rtimes F_1$ is given by (the restriction of) the Artin representation, see, for instance, Birman \cite{Birman}.  This structure is in evidence in the standard presentation of $P_\ell$.  The pure braid group has generators $A_{i,j}$, $1\le i < j \le \ell$, where $F_k=\langle A_{1,k+1},\dots,A_{k,k+1}\rangle$.  The relations in $P_\ell$ are given by 
$A_{r,s}^{-1} A_{i,j}^{} A_{r,s}^{}= \alpha_{j-1}(A_{r,s}^{})(A_{i,j}^{})$, where
\[
\alpha_{j-1}(A_{r,s}^{})(A_{i,j}^{})=
\begin{cases}
A_{i,j}^{}&\text{if $i<r<s<j$ or $r<s<i<j$,}\\
A_{r,j}^{}A_{i,j}^{}A_{r,j}^{-1}&\text{if $r<s=i<j$,}\\
A_{r,j}^{}A_{s,j}^{}A_{i,j}^{}A_{s,j}^{-1}A_{r,j}^{-1}&\text{if $r=i<s<j$,}\\
[A_{r,j}^{},A_{s,j}^{}]A_{i,j}^{}[A_{r,j}^{},A_{s,j}^{}]^{-1}&\text{if $r<i<s<j$.}
\end{cases}
\]
and $[u,v]=uvu^{-1}v^{-1}$ denotes the commutator. 
In the notation established above, the free group $F_k$, $1\le k \le \ell-1$, in the almost-direct product decomposition of $P_\ell$ is generated by $x_{k,i}=A_{i,k+1}$, $1\le i \le k$.
\end{example}

Interest in braid groups and configuration spaces, and generalizations such as  
complements of fiber-type (or supersolvable) hyperplane arrangements and orbit configuration spaces 
has prompted a great deal of work on almost-direct products of free groups, and much is known about the structure of these groups.  For instance, the iterated semidirect product structure of $G=\rtimes_{i=1}^\ell F_{n_i}$ is used in \cite{CSchain} to construct a finite, free, length $\ell$ resolution of the integers over the group ring $\Z{G}$.  Consequently, an arbitrary iterated semidirect product of free groups $G$ is of type $FL$ and has cohomological dimension $\ell$.   In the case where $G$ is an almost-direct product, analysis of this resolution 
reveals that the integral (co)homology groups of $G$ are torsion free, and that the Hibert series of the cohomology ring is given by
\begin{equation} \label{eq:HS}
\hh(H^*(G),t) = \sum_{k=1}^\ell \dim_\Q H^k(G;\Q)\cdot t^k = \prod_{i=1}^\ell (1+n_i t).
\end{equation}

This result may also be obtained using a spectral sequence argument, see Falk and Randell \cite{FR1}.  Furthermore, the techniques of \cite{FR1} may be used to prove that an almost-direct product of free groups $G$ satisfies the famous LCS formula, first established for the pure braid group by Kohno \cite{Kohno}. Let $G_k$ be the $k$-th lower central series subgroup of $G$, defined inductively by $G_1=G$ and $G_{k+1}=[G_k,G]$ for $k \ge 1$.
If $\phi_k=\rank G_k/G_{k+1}$ denotes the rank of the $k$-th lower central series quotient, 
then, in $\Z[[t]]$, one has
\[
\hh(H^*(G),-t) = \prod_{i=1}^\ell (1-n_i t) = \prod_{k\ge 1} (1-t^k)^{\phi_k}.
\]
Additionally, the methods of \cite{FR1,FR2} may be applied to show that an almost-direct product of free groups $G$ is residually nilpotent without torsion, that is, $\cap_{k\ge 1} G_k=\set{1}$ and $G_k/G_{k+1}$ is torsion free for each $k$.  As shown by Paris \cite{paris}, it follows that $G$ is biorderable, and hence the group ring $\Z{G}$ has no zero divisors.

For certain almost-direct products of free groups, the structure of the cohomology ring is known.  In the case where $G$ is the upper triangular McCool group, a subgroup of the group of basis-conjugating automorphisms of the free group $F_n$, the cohomology ring was 
recently 
determined by Cohen, Pakianathan, Vershinin, and Wu \cite{cpvw}.   If $G$ is the fundamental group of the complement of a fiber-type hyperplane arrangement $\A$, the cohomology ring $H^*(G)$ is isomorphic to the well known Orlik-Solomon algebra of $\A$, so is determined by combinatorial aspects of $\A$.  See Orlik and Terao \cite{OT} as a general reference on arrangements.  In particular, the cohomology ring of the pure braid group may be described in this way, recovering a classical result of Arnold \cite{Ar} and Cohen \cite{Co}. 
For any fiber-type arrangement $\A$, Shelton and Yuzvinsky \cite{SY} show that the (rational) cohomology ring $H^*(G)$ is a Koszul algebra.

In this paper, we determine the structure of the cohomology ring of an arbitrary almost-direct product of free groups $G$.  Our results provide an algorithm which takes as input the presentation of $G$ with relations \eqref{eq:rels}, and yields an explicit description of $H^*(G) \cong \sE/\sJ$ as a quotient of the exterior algebra $\sE=\bigwedge H^1(G)$, see Theorem \ref{thm:cohomology}.  This description is used to show that $H^*(G)$ is Koszul for any almost-direct product $G$ in Theorem \ref{thm:koszul}.  As an application, in Theorem \ref{thm:tc}, we compute the topological complexity of the Eilenberg-Mac\,Lane space associated to the group $G \times \Z^m$ for any almost-direct product $G=\rtimes_{i=1}^\ell F_{n_i}$ satisfying $n_i\ge 2$ for each $i$. This homotopy type invariant, introduced by Farber \cite{Fa03}, is motivated by the motion planning problem from robotics.

Some of the results of this paper were announced in \cite{mfo}.

\section{Fox calculus} \label{sec:fox}
Let $G =\rtimes_{i=1}^\ell F_{n_i}$ be an almost-direct product of free groups.  It is not difficult to show that the relations \eqref{eq:rels} may be expressed as commutators.  Consequently, 
the abelianization $H_1(G;\Z)=G/[G,G]$ is free abelian of rank $N=\sum_{i=1}^\ell n_i$.
In this section, we use the Fox calculus to analyze the maps in low-dimensional integral homology and cohomology induced by the abelianization map $\ab\colon G \to G/[G,G] \cong \Z^N$. 

\begin{theorem} \label{thm:h2}
Let $G$ be an almost-direct product of free groups with abelianization $G/[G,G]=\Z^N$. 
For $i \le 2$, the map $\ab_*\colon H_i(G;\Z) \to H_i(\Z^N;\Z)$ in integral homology is injective, and the map $\ab^*\colon H^i(\Z^N;\Z) \to H^i(G;\Z)$ in integral cohomology is surjective.
\end{theorem}

We first exhibit a presentation of $G$ that is particularly amenable to analysis by Fox calculus. 
Let $\IA_n$ denote the kernel of the natural map $\Aut(F_n) \to \GL(n,\Z)$ induced by the map of $F_n$ to its abelianization.  As shown by Magnus and Nielsen, the group $\IA_n$ of $\IA$-automorphisms of $F_n$ is generated by automorphisms $\beta_{i,j}$, $1\le i,j \le n$, $i\neq j$, and $\theta_{i;s,t}$, $1\le i,s,t\le n$, $i,s,t$ distinct, see \cite{mks}. 
If $F_n$ is generated by $y_1,\dots,y_n$, these automorphisms are given by
\begin{equation} \label{eq:IAgens}
\beta_{i,j}(y_k^{})=\begin{cases}
y_k^{} &\text{if $k\neq i$,}\\
y_j^{-1}y_i^{} y_j^{} &\text{if $k=i$,}
\end{cases}
\quad\text{and}\quad
\theta_{i;s,t}(y_k^{})=\begin{cases}
y_k^{} &\text{if $k\neq i$,}\\
y_i^{} [y_s^{},y_t^{}]&\text{if $k=i$.}
\end{cases}
\end{equation}

\begin{proposition} \label{prop:presentation}
Let $G=F_{n_\ell} \rtimes_{\alpha_\ell}\rtimes \cdots 
\rtimes_{\alpha_2} F_{n_1}$ be an almost-direct product of free groups.  Then $G$ admits a presentation with generators $x_{i,p}$, $1\le i \le \ell$, $1\le p \le n_i$, and relations
\[
x_{j,q}^{} x_{i,p}^{} = x_{i,p}^{} x_{j,q}^{} w_{p,q}^{i,j},\quad 1\le i < j \le \ell,
\quad 1\le p \le n_i, \quad 1\le q \le n_j,
\]
where $w_{i,j}^{p,q}$ is a word in the generators $x_{j,1},\dots,x_{j,n_j}$, and is an element of the commutator subgroup $[G,G]$ of $G$.
\end{proposition}
\begin{proof}
The almost-direct of free groups  $G$ admits a presentation with generators $x_{i,p}$ and relations $x_{j,q}^{} x_{i,p}^{} = x_{i,p}^{} \alpha_j(x_{i,p}^{})(x_{j,q}^{})$, where $\alpha_j(x_{i,p}) \in \IA_{n_j}$, see \eqref{eq:rels}.  Thus, $\alpha_j(x_{i,p})=\psi_1^{\epsilon_1}\cdots\psi_m^{\epsilon_m}$, where each $\psi_k$, $1\le k \le m$, is one of the generators $\beta_{i,j}$ and $\theta_{i;s,t}$ of $\IA_{n_j}$ recorded in \eqref{eq:IAgens} above and $\epsilon_k\in\set{1,-1}$.  Clearly, $w_{i,j}^{p,q}=\alpha_j(x_{i,p}^{})(x_{j,q}^{})$ is a word in the generators $x_{j,1},\dots,x_{j,n_j}$.  Observing that
$\beta_{i,j}(y_i^{})=y_i^{}[y_i^{-1},y_j^{-1}]$, $\beta^{-1}_{i,j}(y_i^{})=y_i^{}[y_i^{-1},y_j^{}]$, and $\theta_{i;s,t}^{-1}(y^{}_i)=y_i^{}[y_t^{},y_s^{}]$, induction on $m$ shows that $w_{i,j}^{p,q}$ is a commutator.
\end{proof}

\begin{example} \label{ex:braid2}
In terms of the standard generators $A_{i,j}$ of the pure braid group $P_\ell$, the 
above result yields a presentation with relations
\[
A_{i,j}^{} A_{r,s}^{}=
\begin{cases}
A_{r,s}^{} A_{i,j}^{}&\text{if $i<r<s<j$ or $r<s<i<j$,}\\
A_{r,s}^{}A_{s,j}^{}[A_{s,j}^{-1},A_{r,j}^{}]&\text{if $r<s=i<j$,}\\
A_{r,s}^{}A_{r,j}^{}[A_{s,j}^{},A_{r,j}^{}]&\text{if $r=i<s<j$,}\\
A_{r,s}^{}A_{i,j}^{}[A_{i,j}^{-1},[A_{r,j}^{},A_{s,j}^{}]]&\text{if $r<i<s<j$.}
\end{cases}
\]
\end{example}

Let $F_N$ be the free group on generators $x_1,\dots,x_N$, with integral group ring $\Z{F}_N$.  The standard $\Z{F}_N$-resolution of $\Z$ is given by
\[
(\Z{F}_N)^N \xrightarrow{\ \partial_1\ } \Z{F}_N \xrightarrow{\ \epsilon\ } \Z,
\]
where $(\Z{F}_N)^N$ is a free $\Z{F}_N$-module of rank $n$ with basis $e_1,\dots,e_N$, $\partial_1(e_i)=x_i-1$, and $\epsilon(x_i)=1$.  The Fox calculus is based on the fact that the augmentation ideal $IF_N=\ker\epsilon$ is a free $\Z{F}_N$-module of rank $n$, generated by $\set{x_i-1 \mid 1\le i \le N}$.  In other words, for any $w \in \Z{F}_N$, there are unique elements $\frac{\partial{w}}{\partial{x_i}} \in \Z{F}_N$, the Fox derivatives of $w$, so that 
\begin{equation} \label{eq:FTFC}
w-\epsilon(w) = \sum_{i=1}^N \frac{\partial{w}}{\partial{x_i}}(x_i-1).
\end{equation}

Define the Fox gradient, the $\Z{F}_N$-linear homomorphism $\nabla\colon \Z{F}_N \to (\Z{F}_N)^N$, by
\[
\nabla(w)=\sum_{i=1}^N \frac{\partial{w}}{\partial{x_i}}e_i.
\]
Then, the ``fundamental formula of Fox calculus'' \eqref{eq:FTFC} 
reads 
$w-\epsilon(w)=\partial_1(\nabla(w))$.  This may be used to establish the ``product rule'' $\nabla(uv)=\nabla(u) \cdot\epsilon(v) + u\nabla(v)$.
In particular, if $z\in F_N$, then $\nabla(z^{-1})=-z^{-1}\nabla(z)$.

For a finitely presented group $G=F_N/R$, the Fox calculus may be used to obtain a partial resolution of $\Z$ as a (left) $\Z{G}$-module.  If $R$ is the normal closure of $\set{r_1,\dots,r_M}$ in $F_N$ and $\phi\colon F_N \twoheadrightarrow G$ is the natural projection, with extention $\tilde\phi\colon \Z{F}_N \to\Z{G}$ to group rings, this partial resolution is of the form
\begin{equation}\label{eq:res2}
(\Z{G})^M \xrightarrow{\ \partial_2^G\ } (\Z{G})^N \xrightarrow{\ \partial_1^G\ } \Z{G} 
\xrightarrow{\ \epsilon\ } \Z,
\end{equation}
where $\epsilon$ is the augmentation map, $\partial_1^G=\tilde\phi \circ \partial_1$, and (the matrix of) the map $\partial_2^G$ is given by the matrix of Fox derivatives
\begin{equation*} \label{eq:alex mat}
\left(\tilde\phi\left(\frac{\partial{r}_i}{\partial{x}_j}\right) \right).
\end{equation*}

Now let $G=\rtimes_{i=1}^\ell F_{n_i}$ be an almost-direct product of free groups, let $N=\sum_{i=1}^\ell n_i$, and denote the generators of $F_N$ by $x_{i,q}$, $1\le i \le \ell$, $1\le q \le n_i$, in accordance with the presentation provided by Proposition \ref{prop:presentation}.  A free $\Z{G}$-resolution of $\Z$,
\begin{equation}\label{eq:res}
C_\ell(G) \xrightarrow{\ \partial_\ell^G\ } 
C_{\ell-1}(G) \xrightarrow{\quad} \dots\dots \xrightarrow{\quad}
C_2(G) \xrightarrow{\ \partial_2^G\ } 
C_1(G) \xrightarrow{\ \partial_1^G\ } 
C_0(G) \xrightarrow{\ \epsilon\ } \Z,
\end{equation}
is constructed in \cite{CSchain}.  This resolution is minimal in the sense that $C_q(G)$ is a free (left) $\Z{G}$-module of rank equal to $b_q(G)$, the $q$-th Betti number of $G$.  
That is, 
the boundary maps of this resolution all augment to zero, $\epsilon\circ \partial_q^G = 0$, see \cite[Prop. 3.3]{CSchain}.  Using the construction of \cite{CSchain}, one can show that the truncation $(C_{\le 2}(G),\partial_{\le 2}^G)$ of this resolution coincides with the partial resolution \eqref{eq:res2} obtained by applying the Fox calculus to the presentation of Proposition \ref{prop:presentation} of the almost-direct product of free groups $G$.

The resolution \eqref{eq:res} may be realized as the augmented, cellular chain complex of the universal cover $\widetilde{X}_G$ of a CW-complex $X_G$ of type $K(G,1)$.  See \cite[\S1.3]{CSchain} for the construction of the complex $X_G$.  
As noted above, the abelianization of $G =\rtimes_{i=1}^\ell F_{n_i}$ is free abelian of rank $N=\sum_{i=1}^\ell n_i$.  Denote the generators of $G/[G,G]\cong \Z^N$ by $t_{i,j}$, $1 \le i \le \ell$, $1\le j \le n_i$. The group ring $\Z{\Z^N}$ may be identified with the ring $\Lambda=\Z[t_{i,j}^{\pm 1}]$ of Laurent polynomials.
Let $Y_G$ be the universal abelian cover of $X_G$, the covering corresponding to the abelianization map $\ab\colon G \to \Z^N$.  Denote the cellular chain complex of $Y_G$ by $(\sC_\bullet,\delta_\bullet)$, where $\sC_q = \Lambda \otimes_{\Z{G}} C_q(G)$ and $\delta_q = \id_\Lambda \otimes_{\Z{G}} \partial_q^G$.

Abelianization induces a chain map $\ab_\bullet \colon (\sC_\bullet,{\delta}_\bullet) \to (\sK_\bullet,\sd_\bullet)$, where the latter is the chain complex of the universal (abelian) cover of the $N$-dimensional torus, $(S^1)^{\times N}$, a $K(\Z^N,1)$-space.  Using the standard CW decomposition of the torus, the complex $(\sK_\bullet,\sd_\bullet)$ may be realized as the Koszul complex, with $\sK_1=\Lambda^N$ generated by $e_{i,j}$, 
$\sK_q=\Lambda^{\binom{N}{q}}$ generated by $e_{i_1,j_1}\cdots e_{i_q,j_q}$, and 
\[
\sd_q(e_{i_1,j_1}\cdots e_{i_q,j_q})=\sum_{p=1}^q (-1)^{p+q}(t_{i_p,j_p}-1)e_{i_1,j_1}\cdots 
e_{i_{p-1},j_{p-1}}\cdot e_{i_{p+1},j_{p+1}}
\cdots e_{i_q,j_q}.
\]

The maps $\ab_0\colon \sC_0 \to \sK_0$ and $\ab_1\colon \sC_1 \to \sK_1$ may be taken to be identity maps.  The chain group $\sC_2$ has basis in correspondence with the relations in the presentation of $G$ recorded in Proposition \ref{prop:presentation}.  Let ${\b{r}}_{i,j}^{p,q}$ be the basis element corresponding to the relation $x_{i,p}x_{j,q}=x_{j,q}x_{i,p}w_{i,j}^{p,q}$.  Recall that $w_{i,j}^{p,q}\in [G,G]$ is a commutator in the generators $x_{j,1},\dots,x_{j,n_j}$. We explicitly identify the map $\ab_2\colon \sC_2 \to \sK_2$ (up to chain equivalence).  For this, we use the abelianized Fox gradient, the $\Lambda$-linear homomorphism $\nabla^\ab =\tilde\ab \circ \nabla$.

\begin{proposition} \label{prop:a2}
Let $1\le i < j \le \ell$, $1\le p \le n_i$, and $1\le q \le n_j$.  If $w_{i,j}^{p,q}=
\prod_{k=1}^m [u_k, v_k]$, where $u_k$ and $v_k$ are words in the generators $x_{j,1},\dots,x_{j,n_j}$ of $G$, then
\[
\ab_2(\b{r}_{i,j}^{p,q}) = e_{i,p}e_{j,q}+t_{i,p}t_{j,q}\sum_{k=1}^m\nabla^\ab(u_{k}) \nabla^\ab(v_{k})
\]
\end{proposition}
\begin{proof}
It suffices to check that $\sd_2\circ\ab_2(\b{r}_{i,j}^{p,q}) = \delta_2(\b{r}_{i,j}^{p,q})$.  This is an exercise using the Fox calculus.
\end{proof}

\begin{proof}[Proof of Theorem \ref{thm:h2}]
Abusing notation, let $\epsilon\colon \Lambda \to \Z$ denote the augmentation map, sending a Laurent polynomial to its evaluation at $1$,  $\epsilon(g)=\left. g \right|_{t_{i,j} \mapsto 1}$.  Since the 
boundary maps of the complexes $\sC_\bullet$ and $\sK_\bullet$ both augment to zero, $\epsilon \circ \delta_q=0$ and $\epsilon \circ \sd_q=0$, all homology groups $H_i(G;\Z)$ and $H_i(\Z^N;\Z)$ are torsion free, and the map $\ab_*$ in homology is given simply by $\ab_*=\epsilon \circ \ab_\bullet$.  It follows immediately that $\ab_*\colon H_i(G;\Z)=\Z\otimes_\Lambda \sC_i\to\Z\otimes_\Lambda \sK_i=H_i(\Z^N;\Z)$ is an isomorphism for $i=0,1$.

The bases $\set{\b{r}_{i,j}^{p,q}}$ and $\set{e_{i,p}e_{j,q}}$ for the chain groups $\sC_2$ and $\sK_2$ correspond to bases of the homology groups
$H_2(G;\Z)$ and $H_2(\Z^N;\Z)$, 
which we denote by the same symbols.
By Proposition \ref{prop:a2}, 
\[
\ab_*(\b{r}_{i,j}^{p,q})=\epsilon\circ\ab_2(\b{r}_{i,j}^{p,q}) = e_{i,p}e_{j,q} + \sum_{k=1}^m\epsilon(\nabla^\ab(u_{k})) \epsilon(\nabla^\ab(v_{k})).
\]
Since $u_{k}$ and $v_{k}$ are words in the generators $x_{j,1},\dots,x_{j,n_j}$ of $F_{n_j}$ in the almost-direct product decomposition of $G$, we have 
$\nabla^\ab(u_{k})=\sum_{r=1}^{n_j} g_{k,r} e_{j,r}$
and 
$\nabla^\ab(v_{k})=\sum_{r=1}^{n_j} h_{k,r} e_{j,r}$
for some $g_{k,r},h_{k,r} \in \Lambda$.  It follows that
\begin{equation} \label{eq:a2}
\ab_*(\b{r}_{i,j}^{p,q})=e_{i,p}e_{j,q} +\sum_{1\le r<s\le n_j} c_{i,j}^{p,q,r,s} e_{j,r}e_{j,s}
\end{equation}
for some integers $c_{i,j}^{p,q,r,s}$. Order the bases of 
the homology groups $H_2(G;\Z)$ and $H_2(\Z^N;\Z)$ as follows:
\[
\begin{aligned}
H_2(G;\Z)\colon\quad&\set{\b{r}_{1,2}^{p,q}},\set{\b{r}_{1,3}^{p,q}},\set{\b{r}_{2,3}^{p,q}},\dots,\set{\b{r}_{\ell-1,\ell}^{p,q}},\quad\text{and} \\
H_2(\Z^N;\Z)\colon\quad&\set{e_{1,p}e_{1,q}}, \set{e_{1,p}e_{2,q}}, \set{e_{2,p}e_{2,q}},
\set{e_{1,p}e_{3,q}}, \set{e_{2,p}e_{3,q}}, \set{e_{3,p}e_{3,q}},\dots\dots,\\
&\qquad
\set{e_{1,p}e_{\ell,q}},\dots, \set{e_{\ell-1,p}e_{\ell,q}}, \set{e_{\ell,p}e_{\ell,q}},
\end{aligned}
\]
where each subset is ordered lexicographically (by $\set{p,q}$). With these choices, the matrix of the map $\ab_2\colon H_2(G;\Z) \to H_2(\Z^N;\Z)$ is of the form
\begin{equation} \label{eq:a2matrix}
A=
{\footnotesize{\left(
\begin{array}{cccccccccccc}
0 & I & C_{1,2} & 0 & 0 & 0 & \cdots\cdots\cdots & 0 & \cdots & 0 & 0 \\
0 & 0 & 0 & I & 0 & C_{1,3} & \cdots\cdots\cdots & 0 & \cdots & 0 & 0 \\
0 & 0 & 0 & 0 & I & C_{2,3} & \cdots\cdots\cdots & 0 & \cdots & 0 & 0 \\
\vdots &   &   &   &   &   & \ddots &  &  &  & \vdots \\
0 & 0 & 0 & 0 & 0 & 0 & \cdots\cdots\cdots & I & \cdots & 0 & C_{1,\ell} \\ 
\vdots &  &  &  &  &  &   &   & \ddots &  & \vdots \\ 
0 & 0 & 0 & 0 & 0 & 0 & \cdots\cdots\cdots & 0 & \cdots & I & C_{\ell-1,\ell} 
\end{array}
\right)}},
\end{equation}
where $I$ denotes an identity matrix of appropriate size, and the entries of $C_{i,j}$ are determined by \eqref{eq:a2}.  It follows that $\ab_2\colon H_2(G;\Z) \to H_2(\Z^N;\Z)$ is injective.

Passing to cohomology, the map $\ab^*$ from 
$H^i(\Z^N;\Z)=H_i(\Z^n;\Z)^*$ 
to
$H^i(G;\Z)=H_i(G;\Z)^*$ 
is the dual of $\ab_*\colon H_i(G;\Z) \to H_i(\Z^N;\Z)$, so is surjective for $i\le 2$.
\end{proof}

For brevity, denote the generators of $H^1(\Z^N;\Z) = H_1(\Z^N;\Z)^*$ by the same symbols. Then, the cohomology ring $H^*(\Z^N;\Z)$ is the exterior algebra over $\Z$ generated by $e_{i,p}$, $1\le i \le \ell$, $1\le p \le n_i$.  The proof of Theorem \ref{thm:h2} may be used to explicitly identify the kernel of the map $\ab^*\colon H^2(\Z^N;\Z)
\to H^2(G;\Z)$.  Let 
\[
B=
{\footnotesize{
\left(
\begin{array}{ccccc}
I & 0 & 0 & \cdots & 0\\
0 & K_{1,2} & 0 & \cdots & 0\\
0 & 0 & K_{1,3} & \cdots & 0\\
0 & 0 & K_{2,3} & \cdots & 0\\
0 & 0 & I & \cdots & 0 \\
\vdots & & & \ddots &\vdots \\
0 & 0 & 0 & \cdots & K_{1,\ell}\\
\vdots & \vdots & \vdots & &\vdots \\
0 & 0 & 0 & \cdots & K_{\ell-1,\ell}\\
0 & 0 & 0 & \cdots & I
\end{array}
\right)}}
\]
be the unique integral matrix satisfying $AB=0$, where $A$ is given by \eqref{eq:a2matrix}. 
Note that $K_{i,j}=-C_{i,j}$.  Define elements $\eta_j^{p,q} \in H^2(\Z^N;\Z)$, 
$1\le j \le \ell$, $1\le p<q\le n_j$, corresponding to the columns of $B$, 
\begin{equation} \label{eq:gb}
\eta_j^{p,q} = e_{j,p}e_{j,q}+\sum_{i=1}^{j-1} \sum_{r=1}^{n_i} \sum_{s=1}^{n_j} \kappa_{i,j}^{p,q,r,s} e_{i,r} e_{j,s},
\end{equation}
where the coefficients $\kappa_{i,j}^{p,q,r,s}$ are the entries of the matrices $K_{i,j}$, $1\le i \le j-1$.

\begin{corollary} \label{cor:basis}
The set 
\[
\mathcal{J} = \set{\eta_j^{p,q} \mid 1\le j \le \ell, 1\le p<q\le n_j}
\] 
is a basis for $\ker(\ab^*\colon H^2(\Z^N;\Z)\to H^2(G;\Z))$.
\end{corollary}

\begin{example} \label{ex:braid3}
Let $G=P_\ell=F_{\ell-1} \rtimes_{\alpha_{\ell-1}} \rtimes \cdots \rtimes_{\alpha_2} \rtimes F_1$ be the pure braid group.  Let $N=\binom{\ell}{2}=\sum_{i=1}^{\ell-1} i$, and denote the generators of $H^1(\Z^N;\Z)$ by $e_{i,j}$, $1\le i < j \le \ell$. Using the presentation of $P_\ell$ from Example \ref{ex:braid2}, the above construction yields the basis 
\[
\set{e_{i,j}e_{i,k}-e_{i,j}e_{j,k}+e_{i,k}e_{j,k} \mid 1 \le i < j < k \le \ell}.
\]
for $\ker(\ab^*\colon H^2(\Z^N;\Z) \to H^2(P_\ell;\Z))$.
\end{example}

\section{Cohomology} \label{sec:cohomology}
In this section, we determine the structure of the 
cohomology ring of the almost-direct product of free groups $G=\rtimes_{i=1}^\ell F_{n_i}$.  Since $H^*(G;\Z)$ is torsion free, it suffices to analyze the rational cohomology ring $H^*(G)=H^*(G;\Q)$.  Let $\sE=H^*(\Z^N;\Q)$ be the exterior algebra over $\Q$, generated by $e_{i,p}$, $1\le i \le \ell$, $1\le p \le n_i$.  By Corollary \ref{cor:basis}, the set $\cJ$ is a basis for 
$\ker(\ab^*\colon \sE^2 \to H^2(G))$.  The main results of this section are the following.

\begin{theorem} \label{thm:cohomology}
Let $G=\rtimes_{i=1}^\ell F_{n_i}$ be an almost-direct product of free groups.  
The rational cohomology ring $H^*(G)$ is isomorphic to $\sE/\sJ$, where $\sE$ is the exterior algebra over $\Q$ generated by degree one elements $e_{i,p}$, $1\le i \le \ell$, $1\le p \le n_i$, and $\sJ$ is the homogeneous, two-sided ideal generated by the elements of the set $\cJ$.
\end{theorem}

Recall that a connected, graded algebra $\sA$ over a field $\k$ is said to be a Koszul algebra if $\Tor^{\sA}_{p,q}(\k,\k)=0$ for all $p\neq q$, where $p$ is the homological degree of the $\Tor$ groups, and $q$ is the internal degree coming from the grading of $\sA$.

\begin{theorem} \label{thm:koszul}
Let $G=\rtimes_{i=1}^\ell F_{n_i}$ be an almost-direct product of free groups.  
The rational cohomology ring $H^*(G)$ is a Koszul algebra.
\end{theorem}

To establish these results, we use Gr\"obner basis theory in the exterior algebra.  
Order the generators of $\sE$ as follows:
\[
e_{1,1}<e_{1,2}<\cdots<e_{1,n_1}<e_{2,1}<e_{2,2}<\cdots<e_{2,n_2}<\cdots\cdots <
e_{\ell,1}<e_{\ell,2}<\cdots<e_{\ell,n_\ell}.
\]
If $Q_j=\set{q_1,\dots,q_m}$ is an increasingly ordered subset of $[n_j]=\set{1,\dots,n_j}$, let
\[
e_{Q_j}=\begin{cases}
1&\text{if $Q_j=\emptyset$,}\\
e_{j,q}&\text{if $Q_j=\set{q}$,}\\
e_{j,q_1}e_{j,q_2}e_{j,q_3}\cdots e_{j,q_m}&\text{otherwise.}
\end{cases}
\]
The standard monomials in $\sE$ are elements of the form $e_\bQ=e_{Q_1}e_{Q_2}\cdots e_{Q_\ell}$, where $\bQ=\set{Q_1,Q_2,\dots,Q_\ell}$ and each $Q_j \subset [n_j]$ as above.
The above ordering of the generators of $\sE$ induces the 
deg-lex order 
on the set of all standard monomials. If $\mathbf{P}=\set{P_1,P_2,\dots,P_\ell}$ with $P_j\subset [n_j]$, 
let $|\!|\mathbf{P}|\!|=\sum_{j=1}^\ell |P_j|$.  Then $e_{\mathbf{P}}<e_\bQ$ if $|\!|\mathbf{P}|\!| < |\!|\bQ|\!|$, or $|\!|\mathbf{P}|\!| = |\!|\bQ|\!|$ and there exist $j$, $1\le j \le \ell$, and $k$, $1\le k \le |P_j|$, so that $P_i=Q_i$ for $i<j$, $P_j=\set{p_1,\dots,p_{k-1},p_{k},\dots,p_m}$, $Q_j=\set{p_1,\dots,p_{k-1},q_{k},\dots,q_{m'}}$, and $p_k<q_k$.  The deg-lex order is a linear order on the standard basis $\set{e_\bQ}$ of $\sE$ that is multiplicative in the following sense.  If $e_{\mathbf{P}}$ and $e_{\bQ}$ are nontrivial standard monomials with $e_{\mathbf{P}} e_\bQ \neq 0$, then $e_{\mathbf{P}} e_\bQ$ is a standard monomial up to sign, and $1  < e_{\mathbf{P}} < \pm e_{\mathbf{P}} e_\bQ$.

If $f=\sum c_\bQ e_\bQ$ is an arbitrary element of $\sE$, the initial term $\inn(f)$ of $f$ is the term $c_\bQ e_\bQ$ for which $e_\bQ$ is the largest monomial for all $\bQ$ with $c_\bQ\neq 0$.  If $\sI \subset \sE$ is an ideal of $\sE$, the initial ideal $\inn(\sI)$ of $\sI$ is the ideal generated by the initial terms $\inn(f)$, $f \in \sI$.  A set of elements $f_1,\dots,f_m \in \sI$ is a Gr\"obner basis for $\sI$ if the initial ideal $\inn(\sI)$ is generated by $\inn(f_1),\dots,\inn(f_m)$.

\begin{lemma} \label{lem:g-basis}
The set $\cJ$ is a Gr\"obner basis for the ideal $\sJ$.
\end{lemma}
\begin{proof}
If $Q_j=\set{q_1,\dots,q_m}$ is an increasingly ordered subset of $[n_j]=\set{1,\dots,n_j}$, let
\[
\xi_{Q_j}=\begin{cases}
1&\text{if $Q_j=\emptyset$,}\\
e_{j,q}&\text{if $Q_j=\set{q}$,}\\
\eta_{j}^{q_1,q_2} e_{j,q_3}\cdots e_{j,q_m}&\text{otherwise,}
\end{cases}
\]
where $\eta_j^{p,q} \in \sJ$ is the element of $\cJ$ given by \eqref{eq:gb}.  Note that $\xi_{Q_j} \in \sJ$ if $|Q_j| \ge 2$ and $\xi_{Q_j} \notin \sJ$ if $|Q_j|\le 1$.  If $\bQ=\set{Q_1,\dots,Q_\ell}$, 
where $Q_j \subset [n_j]$ as above for each $j$, define
\[
\xi_\bQ = \xi_{Q_1}\xi_{Q_2}\cdots \xi_{Q_\ell}.
\]

As noted above, the set $\set{e_\bQ}$, for all possible choices of $\bQ$, is the standard basis for the exterior algebra $\sE$.  It is readily checked that the set $\set{\xi_\bQ}$ (again, for all possible choices of $\bQ$) is also a basis for $\sE$.  One can show, for instance, that the map $\psi\colon \sE \to \sE$ defined by $\psi(e_\bQ) =\xi_\bQ$ is an isomorphism (of vector spaces).

To show that $\cJ$ is a Gr\"obner basis for the ideal $\sJ$, it suffices to show that $\sJ$ and the ideal ${\sf{I}}=\langle \inn(\eta_j^{p,q}) \mid 1\le j \le \ell, 1\le p<q\le n_j\rangle$ generated by the initial terms of the elements of $\cJ$ have the same Hilbert function, see \cite[Cor. 1.2]{AHH}.  Since $\set{\xi_\bQ}$ is a basis for $\sE$ and $\xi_{Q_j}\in \sJ$ if $|Q_j|\ge 2$, the set
\[
\set{\xi_\bQ = \xi_{Q_1}\xi_{Q_2}\cdots \xi_{Q_\ell} \mid |Q_j| \ge 2\ \text{for some $j$, $1\le j \le \ell$}}
\]
is a basis for $\sJ$.  The initial term of $\eta_j^{p,q}$ in the deg-lex order is given by $\inn(\eta_j^{p,q})=e_{j,p}e_{j,q}$.  Consequently, the set 
\[
\set{e_\bQ = e_{Q_1}e_{Q_2}\cdots e_{Q_\ell} \mid |Q_j| \ge 2\ \text{for some $j$, $1\le j \le \ell$}}
\]
is a basis for ${\sf{I}}$.  It follows immediately that the ideals ${\sf{I}}$ and $\sJ$ have the same Hilbert function.
\end{proof}

We now establish the main results of this section.

\begin{proof}[Proof of Theorem \ref{thm:cohomology}]
Let $G=\rtimes_{i=1}^\ell F_{n_i}$ be an almost-direct product of free groups.  We first show that $H^*(G)=H^*(G;\Q)$ is generated in degree one.  This is clear if the cohomological dimension of $G$ is equal to one. 

Consider the split, short exact sequence of groups $1 \to F_{n_\ell} \to G \to \rtimes_{i=1}^{\ell-1} F_{n_i}\to 1$, and the corresponding fibration $F \xrightarrow{\iota}E \xrightarrow{\rho} B$ of Eilenberg-Mac\,Lane spaces, with fiber $F=\bigvee_{n_\ell} S^1$, a bouquet of $n_\ell$ circles.  Since $G$ is an almost-direct product, the group $\bar{G}=\rtimes_{i=1}^{\ell-1} F_{n_i}=\pi_1(B)$ acts trivially on the cohomology $H^*(F_{n_\ell})=H^*(F)$.  Checking that the map $\iota^* \colon H^1(G)=H^1(E) \to H^1(F)=H^1(F_{n_\ell})$ is surjective, by the Leray-Hirsch Theorem (see \cite[Thm. 5.10]{McCleary}), we have an isomorphism of vector spaces
\[
H^*(G) \cong H^*(\bar{G}) \otimes_\Q H^*(F_{n_\ell})
\]

The group $\bar{G}$ has cohomological dimension $\ell-1$. So we may inductively assume that $H^*(\bar{G})$ is generated in degree one.  Using this, the fact that $H^*(F_{n_\ell})$ is also generated in degree one, and the above isomorphism, we conclude that $H^*(G)$ is  generated in degree one as asserted.

By Theorem \ref{thm:h2}, the map $\ab^*\colon \sE^1 \to H^1(G)$ is an isomorphism.  This, together with the above considerations, implies that $\ab^*\colon \sE \to H^*(G)$ is a surjection of algebras.  Thus, $H^*(G) \cong \sE/\ker(\ab^*)$.  Since $\sJ=\ker(\ab^*\colon \sE^2 \to H^2(G))$, we complete the proof by showing that $H^*(G)$ and $\sE/\sJ$ have the same Hilbert series.  As noted in \eqref{eq:HS}, the Hilbert series of $H^*(G)$ is given by $\hh(H^*(G),t)=\prod_{i=1}^\ell(1+n_i t)$.

The proof of Lemma \ref{lem:g-basis} implies that the quotient $\sA=\sE/\sJ$ has a basis with elements in correspondence with those $\xi_\bQ=\xi_{Q_1}\xi_{Q_2}\cdots\xi_{Q_\ell}$ with $|Q_j|\le 1$ for each $j$, $1\le j \le \ell$.  It follows that the summand $\sA^k$ of all degree $k$ homogeneous elements is a vector space of dimension $\sum n_{p_1}n_{p_2}\cdots n_{p_k}$, the sum over all $1\le p_1 <p_2<\cdots<p_k \le \ell$. Consequently, the Hilbert series $\hh(\sA,t)=\prod_{i=1}^\ell(1+n_i t)$ is equal to that of $H^*(G)$.
\end{proof}

\begin{remark} \label{rem:basis}
The above argument yields an explicit basis for $H^*(G) \cong \sA$.  
For $\xi \in \sE$, denote the image of $\xi$ under the natural projection ${\sf{p}}\colon\sE \to \sA=\sE/\sJ$ by $\bar{\xi}={\sf{p}}(\xi)$. Then $\sA$ has basis
\[
\set{\bar{\xi}_\bQ \mid \bQ=\set{Q_1,\dots,Q_\ell}\ \text{and}\ |Q_j|\le 1\ \text{for each $j$, $1\le j \le \ell$}}.
\]
\end{remark}

\begin{proof}[Proof of Theorem \ref{thm:koszul}]
By Theorem \ref{thm:cohomology}, $H^*(G) \cong \sE/\sJ$ is the quotient of an exterior algebra by a homogeneous ideal generated in degree two.  Since $\sJ$ has a quadratic Gr\"obner basis by Lemma \ref{lem:g-basis}, $H^*(G)$ is Koszul (see, for instance, \cite[Thm. 6.16]{Yuz}).
\end{proof}

\begin{example} \label{ex:braid4}
In the case where $G=P_\ell$ is the pure braid group, Theorem \ref{thm:cohomology} shows that the cohomology ring $H^*(P_\ell)$ is generated by degree one elements $e_{i,j}$, $1\le i<j\le \ell$, which satisfy (only) the relations
\[
e_{i,j}e_{i,k}-e_{i,j}e_{j,k}+e_{i,k}e_{j,k}=0\quad\text{for}\quad  1 \le i < j < k \le \ell
\]
and their consequences, see Example \ref{ex:braid3}.  This recovers the classical description of $H^*(P_\ell)$ due to Arnold \cite{Ar} and Cohen \cite{Co}. 

The Koszulity of $H^*(P_\ell)$ ensured by Theorem \ref{thm:koszul} is a consequence of work of Kohno \cite{Kohno}, see also Shelton and Yuzvinsky \cite{SY}.
\end{example}

\section{Topological complexity} \label{sec:tc}
Let $X$ be a path-connected topological space.  We will focus on the case where $X$ is an Eilenberg-Mac\,Lane space of type $K(G,1)$ for an almost-direct product of free groups $G$, so assume that $X$ has the homotopy type of a finite CW-complex.  Viewing $X$ as the space of configurations of a mechanical system, the motion planning problem consists of constructing an algorithm which takes as input pairs of configurations $(x_0,x_1) \in X \times X$, and produces a continuous path $\gamma\colon [0,1] \to X$ from the initial configuration $x_0=\gamma(0)$ to the terminal configuration $x_1=\gamma(1)$.  The motion planning problem is of interest in robotics, see, for example, Latombe \cite{La} and Sharir \cite{Sh}.

A topological approach to this problem was recently developed by Farber, see the survey \cite{Fa05}.  Let $PX$ denote the space of all continuous paths $\gamma\colon [0,1] \to X$, equipped with the compact-open topology.  The map $\pi\colon PX \to X \times X$, $\gamma \mapsto (\gamma(0),\gamma(1))$, defined by sending a path to its endpoints is a fibration, with fiber $\Omega{X}$, the based loop space of $X$. The motion planning problem asks for a section of this fibration, a map $s\colon X\times X \to PX$ satisfying $\pi \circ s = \id_{X\times X}$.

It would be desirable for the motion planning algorithm to depend continuously on the input.  However, one can show that there exists a globally continuous section $s\colon X\times X \to PX$ if and only if $X$ is contractible, see \cite[Thm.~1]{Fa03}. 

\begin{definition}
The \emph{topological complexity} of $X$, $\tc(X)$,  is the smallest positive integer $k$ for which $X\times X=U_1\cup\dots\cup U_k$, where $U_i$ is open and there exists a continuous section $s_i\colon U_i\to PX$, $\pi \circ s_i=\id_{U_i}$, for each $i$, $1\le i \le k$.  In other words, the topological complexity of $X$ is
the sectional category (or Schwarz genus) of the path space fibration, $\tc(X)=\secat(\pi\colon PX \to X\times X)$.
\end{definition}

Observe that the topological complexity of $X$ is a homotopy type invariant.  If $G$ is a discrete group, define $\tc(G)$, the topological complexity of $G$, to be that of an Eilenberg-Mac\,Lane space of type $K(G,1)$.  In \cite[\S{31}]{Fa05}, Farber poses the problem of determining the topological complexity of $G$ in terms of other invariants of $G$, such as the cohomological dimension, $\dim(G)$.  In this section, we solve this problem for a large class of almost-direct products of free groups.

\begin{theorem} \label{thm:tc}
Let $G = F_{n_\ell} \rtimes \dots \rtimes F_{n_1}$ be an almost-direct product of free groups.  
If $n_j \ge 2$ for each $j$ and $m$ is a non-negative integer, 
then $\tc(G \times \mathbb Z^m)=2\ell+m+1$.  
\end{theorem}

This result is a 
consequence of Theorem \ref{thm:cohomology}, together with known properties of topological complexity. We record the requisite properties before giving the proof.

First, if $X$ is a finite dimensional cell complex as above, then $\tc(X) \le 2\dim(X)+1$, see \cite[Thm. 14.1]{Fa05}.  Recall that the geometric dimension, $\geomdim(G)$, of a group $G$ is the smallest dimension of an Eilenberg-Mac\,Lane complex of type $K(G,1)$.

\begin{lemma} \label{lem:cd=gd}
If $G$ is an almost direct product of free groups, the geometric dimension of $G$ is equal to the cohomological dimension of $G$, $\dim(G) = \geomdim(G)$.
\end{lemma}
\begin{proof}
For an arbitrary iterated semidirect product of finitely generated free groups $G=F_{n_\ell} \rtimes \dots \rtimes F_{n_1}$ of cohomological dimension $\ell$, a $K(G,1)$-complex of dimension $\ell$ is constructed in \cite[\S1.3]{CSchain}.
\end{proof}
Note that this lemma follows from a classical result of Eilenberg and Ganea \cite{EG} in the case where $\dim(G)\ge 3$.  If $G=F_n \rtimes F_m$, the cell complex of \cite[\S1.3]{CSchain} is the ``presentation $2$-complex'' associated the the presentation of Proposition \ref{prop:presentation}.  For an iterated semidirect product of free groups $G$, this lemma and the dimensional upper bound noted above yield
\begin{equation} \label{eq:upper bound}
\tc(G) \le 2 \dim(G) + 1.
\end{equation}

Next, the topological complexity of a product space admits the  
upper bound 
\begin{equation*} 
\tc(X\times Y) \le \tc(X) + \tc(Y) - 1,
\end{equation*}
see \cite[\S{12}]{Fa04}.  Consequently, if $G_1$ and $G_2$ are groups (of finite cohomological dimension), then 
\begin{equation} \label{eq:product}
\tc(G_1 \times G_2) \le \tc(G_1) + \tc(G_2)-1.
\end{equation}

Finally, the sectional category of an arbitrary fibration admits a cohomological lower bound.  If $A=\bigoplus_{j=0}^\ell A^j$ is a graded algebra over a field $\k$, with $A^j$ finite-dimensional for each $j$, define $\cl(A)$, the cup length of $A$, to be the largest integer $q$ for which there are homogeneous elements $a_1,\dots,a_q$ of positive degree in $A$ such that $a_1\cdots a_q \neq 0$. 
If $p\colon E \to B$ is a fibration, the sectional category admits the 
lower bound 
\[
\secat(p\colon E \to B) > \cl\bigl(\ker(p^*\colon H^*(B;\k) \to H^*(E;\k)\bigr),
\]
see \cite[\S{8}]{Ja}. We will work with rational coefficients, and write $H^*(Y)=H^*(Y;\Q)$.

For the path space fibration $\pi \colon PX \to X \times X$, we have
\begin{equation*}
\begin{aligned}
\tc(X)=\secat(\pi \colon PX \to X \times X) &> 
\cl\bigl(\ker(p^*\colon H^*(X\times X) \to H^*(PX)\bigr)\\&=
\cl\bigl(\ker(H^*(X)\otimes H^*(X) \xrightarrow{\ \cup\,} H^*(X)\bigr),
\end{aligned}
\end{equation*}
using the K\"unneth formula and the fact that $PX \simeq X$, see \cite[Thm.~7]{Fa03}.  That is, the topological complexity of $X$ is greater than the {zero divisor cup length}, $\zcl(H^*(X))$, the cup length of the ideal $Z=\ker(H^*(X)\otimes H^*(X) \xrightarrow{\ \cup\,} H^*(X))$ of zero divisors.  In terms of groups, this lower bound may be stated as
\begin{equation} \label{eq:zcl bound}
\tc(G) > \zcl(H^*(G)).
\end{equation}

We now establish the main result of this section.

\begin{proof}[Proof of Theorem \ref{thm:tc}]
Let $G=F_{n_\ell}\rtimes \dots \rtimes F_{n_1}$ be an almost-direct product of free groups with $n_j \ge 2$ for each $j$, and let $m$ be a nonnegative integer.  The topological complexity of the $m$-dimensional torus $(S^1)^{\times m}$ is equal to $m+1$, see \cite[Thm. 13]{Fa03}.  Since this torus is a $K(\Z^m,1)$-space, we have $\tc(\Z^m)=\tc((S^1)^{\times  m})=m+1$.  The product inequality \eqref{eq:product}
 and dimensional upper bound \eqref{eq:upper bound} yield
\[
\tc(G\times \Z^m) \le \tc(G)+\tc(\Z^m)-1\le (2 \ell+1)+(m+1) -1= 2\ell + m + 1.
\]
In light of the lower bound \eqref{eq:zcl bound}, it suffices to show that $\zcl(H^*(G\times \Z^m)) \ge 2\ell+m$.

By the K\"unneth formula, we have $H^*(G\times \Z^m)=H^*(G) \otimes H^*(\Z^m)$ (recall that we use rational coefficients).  The cohomology of $\Z^m$ is an exterior algebra, generated by degree one elements $z_1,\dots,z_n$.  For each $i$, let $\hat{z}_i = 1\otimes z_i-z_i\otimes 1 \in H^*(\Z^m)\otimes H^*(\Z^m)$.  Observe that $\hat{z}_i$ is a zero-divisor.  The product $\hat{z}_1\cdot \hat{z}_2\cdots \hat{z}_m$ is nonzero.  In fact, one has
\begin{equation} \label{eq:prod1}
\hat{z}_1\cdot\hat{z}_2\cdots\hat{z}_m = \sum_{(I,I')} (-1)^{|I|}\mathrm{sign}(\sigma)\, z_I \otimes z_{I'},
\end{equation}
where the sum is over all partitions $(I,I')$ by ordered subsets of $[z]=\set{1,\dots,m}$, 
$z_I=z_{i_1}\cdots z_{i_k}$ if $I=(i_1,\dots,i_k)$, 
and $\sigma$ is the shuffle on $[m]$ which puts every element of $I'$ after all elements of $I$, preserving the orders inside $I$ and $I'$, see \cite[Lemma 10]{FY}.  This, together with the fact that 
$\tc(\Z^m)=m+1$, implies that $\zcl(H^*(\Z^m))=m$.

Since, clearly, $\zcl(H^*(G) \otimes H^*(\Z^m)) \ge \zcl(H^*(G)) + \zcl(H^*(\Z^m))$, it remains to show that $\zcl(H^*(G)) \ge 2\ell$.  
For each $i$, $1\le i \le \ell$, let $x_i=e_{i,1}$ and $y_i=e_{i,2}$ be classes in $H^1(G)$ corresponding to distinct generators of the free group $F_{n_i}$.  
As above, consider the zero divisors $\hat{x}_i=1\otimes x_i-x_i\otimes 1$ and 
$\hat{y}_i=1\otimes y_i-y_i\otimes 1$.
We will show that the product 
\begin{equation} \label{eq:Zproduct}
\prod_{i=1}^\ell\hat{x}_i \hat{y_i}=\prod_{i=1}^\ell (1\otimes x_i-x_i\otimes 1)(1\otimes y_i-y_i\otimes 1)
\end{equation}
is non-zero in $H^*(G) \otimes H^*(G)$.  

By Theorem \ref{thm:cohomology}, $H^*(G) \cong \sE/\sJ$, where $\sE$ is the exterior algebra on $H^1(G)$, and $\sJ$ is the ideal generated by the elements $\eta_j^{p,q}$ recorded in \eqref{eq:gb}.  We first consider the product \eqref{eq:Zproduct} in $\sE \otimes \sE$.
For $1\le k \le \ell$, if $I$ is an ordered subset of $[k]=\set{1,\dots,k}$, 
define $X[I,k]=u_1\cdots u_k$, where $u_i=x_i$ if $i \in I$ and $u_i=y_i$ if $i \notin I$. 
For each such $k$, consider the element $Z_k=\prod_{i=1}^k \hat{x}_i \hat{y}_i$ in $\sE\otimes\sE$.  Let $\sJ'$ be the ideal in $\sE\otimes \sE$ generated by $\set{\eta\otimes 1,1\otimes \eta \mid \eta \in \sJ}$. 

\begin{claim} For each $k$, $1\le k \le \ell$, 
\[
Z_k=\prod_{i=1}^k\hat{x}_i \hat{y}_i=
\epsilon_k\sum_{I \subset[k]} (-1)^{|I|} X[I,k] \otimes Y[I,k] + \omega_k,
\]
where $\epsilon_k=(-1)^{\lfloor\frac{k}{2}\rfloor}$, 
$Y[I,k]=X[[k]\setminus I,k]$ and $\omega_k \in \sJ'$.
\end{claim}
The proof of the claim is by induction on $k$.  For $k=1$, since multiplication in $\sE\otimes \sE$ is given by 
\[
(a\otimes b)\cdot (c\otimes d)=(-1)^{|b|\cdot |c|} ac \otimes bd,
\] 
where $|u|$ denotes the degree of $u$, we have
\begin{equation} \label{eq:xy}
\hat{x}_j \hat{y}_j=y_j\otimes x_j-x_j\otimes y_j + x_j y_j \otimes 1+1\otimes x_j y_j
\end{equation}
for each $j$.  
Since $x_1y_1=e_{1,1}e_{1,2}=\eta_1^{1,2} \in \sJ$,  
$X(\emptyset,1)=y_1$, $X[[1],1]=x_1$, and $\epsilon_1=1$, the claim holds for $Z_1=\hat{x}_1\hat{y}_1$.

Inductively assume that $Z_{k-1}=\prod_{i=1}^{k-1}\hat{x}_i \hat{y}_i$ is as asserted, and consider
\begin{equation} \label{eq:Zk}
Z_{k}= Z_{k-1} \hat{x}_{k} \hat{y}_{k}
=\Bigl(\epsilon_{k-1}\sum_{I \subset[k-1]} (-1)^{|I|} X[I,k-1] \otimes Y[I,k-1] + \omega_{k-1} \Bigr)\hat{x}_{k} \hat{y}_{k}.
\end{equation}
Since $\omega_{k-1} \in \sJ'$ by induction, we have $\omega_{k-1} \hat{x}_{k} \hat{y}_{k}\in \sJ'$.  A straightforward calculation reveals that the sum
\[
\epsilon_{k}\sum_{J \subset[k]} (-1)^{|J|} X[J,k] \otimes Y[J,k]
\]
is equal to 
\[
\Bigl(\epsilon_{k-1}\sum_{I \subset[k-1]} (-1)^{|I|} X[I,k-1] \otimes Y[I,k-1]\Bigr)
(y_{k}\otimes x_{k}-x_{k}\otimes y_{k}).
\]
Using \eqref{eq:Zk} and \eqref{eq:xy}, to establish the claim, it remains to show that 
$X[I,k-1] x_k y_k \in \sJ$ for $I \subset [k-1]$.  By \eqref{eq:gb}, we have
\[
x_{k} y_{k} =e_{k,1}e_{k,2}= \eta_{k}^{1,2}-\sum_{i=1}^{k-1} \sum_{r=1}^{n_i} \sum_{s=1}^{n_{k}} \kappa_{i,k}^{1,2,r,s} e_{i,r} e_{k,s}.
\]
Since $X[I,k-1]=u_1\cdots u_{k-1}=e_{1,q_1} \cdots e_{k-1,q_{k-1}}$, where $q_j \in \set{1,2}$, it suffices to show that $e_{1,q_1} \cdots e_{k-1,q_{k-1}}e_{i,r} e_{k,s} \in \sJ$ for $1\le i \le k-1$, $1\le r\le n_i$, and $1\le s \le n_k$. Since $e_{1,p}e_{1,q} \in \sJ$ for $1\le p<q\le n_1$, this may be accomplished by repeated use of \eqref{eq:gb} as above, completing the proof of the claim.
 
Now consider the product \eqref{eq:Zproduct} in $H^*(G) \otimes H^*(G)=\sA \otimes \sA$, where $\sA=\sE/\sJ$. If $\xi \in \sE$, recall that $\bar{\xi}$ denotes the image of $\xi$ under the natural projection $\sf{p}\colon \sE \to \sA$. From the claim, we obtain
\[
\bar{Z_\ell}=\epsilon_\ell\sum_{I \subset[\ell]} (-1)^{|I|} \bar{X}[I,\ell] \otimes \bar{Y}[I,\ell]
\]
in $H^*(G) \otimes H^*(G)$. 
Using Remark \ref{rem:basis} to check that the set $\set{\bar{X}[I,\ell] \mid I \subset [\ell]}$ is linearly independent in $H^\ell(G)$, we conclude that the product $\prod_{i=1}^\ell\hat{x}_i \hat{y_i}$ is non-zero in $H^*(G) \otimes H^*(G)$.  
So we have 
\[
2\ell \le \zcl(H^*(G)) < \tc(G) \le 2\dim(G)+1 = 2\ell+1.
\]  
Thus, $\tc(G)=2\ell+1$ and $\tc(G\times \Z^m)=2\ell+m+1$.
\end{proof}

\begin{corollary} \label{cor:zcl}
Let $G = F_{n_\ell} \rtimes \dots \rtimes F_{n_1}$ be an almost-direct product of free groups.  
If $n_j \ge 2$ for each $j$ and $m$ is a non-negative integer, 
then $\zcl(H^*(G \times \mathbb Z^m))=2\ell+m$.  
\end{corollary}

\begin{example} \label{ex:braid5}
Let $G=P_\ell$ be the pure braid group, with center $Z(P_\ell)$.  It is well known that $Z(P_\ell)=\Z$ is infinite cyclic, that $\bar{P}_\ell=P_\ell/Z(P_\ell) = \rtimes_{i=2}^{n-1} F_i$ 
is an almost-direct product of free groups, and that $P_\ell \cong \bar{P}_\ell \times \Z$.  Theorem \ref{thm:tc} yields $\tc(\bar{P_\ell})=2\ell-3$ and $\tc(P_\ell)=2\ell-2$, the latter recovering the calculation of the topological complexity of the configuration space of $\ell$ ordered points in $\C$ due to Farber and Yuzvinsky \cite{FY}.

More generally, let $P_{\ell,k}=\ker(P_{k+\ell} \to P_k)$ be the kernel of the homomorphism that forgets the last $\ell\ge 1$ strands of a pure braid.  This group may be realized as the fundamental group of the configuration space $F(\C_k,\ell)$ of $\ell$ ordered points in $\C_k=\C\setminus\set{k\ \text{points}}$, and is an almost-direct product of free groups, $P_{\ell,k}=\rtimes_{i=k}^{k+\ell-1} F_i$.  Since $F(\C_k,\ell)$ is a $K(P_{\ell,k},1)$-space, Theorem \ref{thm:tc} implies that $\tc(F(\C_k,\ell))=\tc(P_{\ell,k})=2\ell+1$ if $k \ge 2$, 
as first shown by Farber, Grant, and Yuzvinsky \cite{fgy}.
\end{example}

The pure braid group and the group $P_{\ell,k}$ may be realized as fundamental groups of complements of fiber-type hyperplane arrangements.    
For an arbitrary fiber-type arrangement $\A$ in $\C^\ell$, the complement $M=\C^\ell\setminus \bigcup_{H\in\A}H$ is a $K(G,1)$-space, and the fundamental group $G=\pi_1(M)=\rtimes_{i=1}^\ell F_{n_i}$ is an almost-direct product of free groups.  Call the integers $n_1,\dots,n_\ell$ the exponents of $\A$.

\begin{corollary}
 Let $G$ be the fundamental group of the complement of a fiber-type hyperplane arrangement $\A$.  If the exponents of $\A$ are all at least $2$, then \[\tc(G\times \Z^m)=2\dim(G)+m+1.                                                                                                          \]\end{corollary}
This may also be obtained using results of Farber and Yuzvinsky \cite{FY}.

We conclude with a final example.

\begin{example}
 The basis-conjugating automorphism group $\PS_n$ of the free group $F_n$ is the subgroup of $\IA_n<\Aut(F_n)$ generated by the automorphisms $\beta_{i,j}$ recorded in \eqref{eq:IAgens}.  The subgroup of $\PS_n$ generated by the automorphisms $\beta_{i,j}$ with $1\le i<j \le n$ is known as the upper-triangular McCool group, and is an almost-direct product of free groups, see \cite{cpvw}. If  $x_{i,p}=\beta_{n-i,n-p+1}$, then $\PS_n^+=\rtimes_{i=1}^{n-1} F_i$, where $F_i=\langle x_{i,1},\dots,x_{i,i}\rangle$.
The presentation of $\PS_n^+$ provided by Proposition \ref{prop:presentation} has relations
\[
 x_{j,q}x_{i,p}=\begin{cases}
                 x_{i,p}^{}x_{j,q}^{}[x_{j,q}^{-1},x_{j,p}^{}]&\text{if $q=i+1$,}\\
                 x_{i,p}^{}x_{j,q}^{}&\text{otherwise,}\\
                \end{cases}
\]
where $1 \le i<j\le n-1$, compare \cite{cp,cpvw}.

Theorem \ref{thm:cohomology} reveals that $H^*(\PS_n^+) \cong \sE/\sJ$, where $\sE$ is the exterior algebra generated by $e_{i,p}$, $1\le p \le i \le n-1$, and $\sJ$ is the ideal generated by $e_{i,p}e_{j,i+1}-e_{j,p}e_{j,i+1}$, $1\le p \le i < j \le n-1$. 
It is readily checked that this differs from the description of $H^*(\PS_n^+)$ given in \cite{cpvw} only by a change in indexing. 
By Theorem \ref{thm:koszul},  $H^*(\PS_n^+)$ is Koszul.  This was first established in \cite{cp} by other means.

In \cite[Prop. 2.3]{cp}, it is shown that the center $Z(\PS_n^+)$ of $\PS_n^+$ is infinite cyclic, that 
$\overline{\PS}_n^+ = \PS_n^+/Z(\PS_n^+) = \rtimes_{i=2}^{n-1} F_i$ is an almost-direct product of free groups, and that $\PS_n^+ \cong \overline{\PS}_n^+ \times \Z$. 
Theorem \ref{thm:tc} yields $\tc(\overline{\PS}_n^+)=2\ell-3$ and $\tc(\PS_n^+)=2\ell-2$, as first shown in \cite{cp}.
\end{example}

\newcommand{\arxiv}[1]{{\texttt{\href{http://arxiv.org/abs/#1}{{arXiv:#1}}}}}

\newcommand{\MRh}[1]{\href{http://www.ams.org/mathscinet-getitem?mr=#1}{MR#1}}

\end{document}